\documentclass{amsart}
\UseRawInputEncoding

\makeatletter
 \def\LaTeX{\leavevmode L\raise.42ex
   \hbox{\kern-.3em\size{\sf@size}{0pt}\selectfont A}\kern-.15em\TeX}
\makeatother

\newcommand{\BibTeX}{{\rm B\kern-.05em{\sc
i\kern-.025emb}\kern-.08em\TeX}}
\newtheorem{col}{Corollary}[section]
\newtheorem{thm}{Theorem}[section]
\newtheorem{defn}{Definition}

\theoremstyle{defn}

\newtheorem{remark}[thm]{Remark}
\newtheorem{lemma}[thm]{Lemma}

\numberwithin{equation}{section}

\begin{document}
\begin{center}
\title[Bernstein spaces, sampling and  interpolation  in Mellin Analysis]{Bernstein spaces, sampling, and  Riesz-Boas interpolation formulas in Mellin Analysis}

\maketitle

\author{Isaac Z. Pesenson }\footnote{ Department of  Mathematics, Temple University,
 Philadelphia,
PA 19122; pesenson@temple.edu  }

%
\end{center}

\begin{abstract}
The goal of the paper is to consider Bernstein-Mellin subspaces in the Lebesgue-Mellin spaces and establishing for functions in these subspaces   new sampling theorems and Riesz-Boas high-order interpolation formulas.  
\end{abstract}

\section{Introduction}
\label{sec:1}
In a series of interesting papers by C. Bardaro at al \cite{BBMS-1}-\cite{BBMS-8}, and also by P. Butzer and S. Jansche \cite{BJ-1}-\cite{BJ-3} authors developed in the framework of Mellin analysis analogs of such important topics as Sobolev spaces, Bernstein spaces, Bernstein inequality, Paley-Wiener theorem, Riesz-Boas interpolation formulas, different sampling results. Many of their results were obtained by using the notion of polar-analytic functions developed  in  \cite{BBMS-1}-\cite{BBMS-4}. 

The objective of the present paper is to present a very different approach to the same topics  based solely on the fact that the family of Mellin translations defined as 
\begin{equation}\label{U}
U_{c}(t)f(x)=e^{ct}f(e^{t}x),\>\>\>U_{c}(t+\tau)=U_{c}(t)U_{c}(\tau),\>\>\>c\in \mathbb{R},
\end{equation}
forms a one-parameter $C_{0}$-group of isometries in appropriate function spaces (see below).
As one can see, its infinitesimal generator is the operator
\begin{equation}
\frac{d}{dt}U_{c}(t)f(x)|_{t=0}=x\frac{d}{dx}f(x)+cf(x)=\Theta_{c}f(x).
\end{equation}
In this paper we are guided by our  abstract theory of  sampling and interpolation in Banach spaces which was developed in \cite{Pes15a}, \cite{Pes15b}.  At the same time, our approach is very specific  and direct and we are not using the language of one-parameter groups.  The fact that we consider a very concrete situation allows us to obtain  results which we did not have in our general development. All our results hold true for a general group of translations $U_{c}$ with any $c\in \mathbb{R}$. However, for the sake of simplicity we consider only the case $c=0$ and adapting notations $U_{0}=U,\>\Theta_{0}=\Theta$.

In section \ref{Bern-sp} we define analog of Bernstein spaces using Bernstein-type inequality for the operator $\Theta$. Our analog of the Paley-Wiener Theorem is Theorem \ref{PW}. In section \ref{sampling} we prove four sampling theorems: Theorem \ref{sampling1}-Theorem \ref{sampling4}.  Our formula (\ref{VT-1}), which is a generalization of the Valiron-Tschakaloff sampling theorem, looks exactly like one proved in \cite{BBMS-1}, Th. 6. Other three theorems 
in this section seems to be new.  They all  deal with the regularly spaced sampling points. In contrast, section \ref{irreg} contains two sampling theorems which are using irregularly spaced sampling points. One of these theorems is a generalization of a sampling theorem which belongs to J.R. Higgins \cite{Hig} and another one to C. Seip \cite{S}.

In section \ref{RB} we discuss what we call the Riesz-Boas interpolation formulas.
The famous Riesz interpolation formula \cite{R1}, \cite{R2}, \cite{Nik}  gives expression of the derivative of a trigonometric polynomial as a linear combination of its translates: 
\begin{equation}\label{Riesz}
\left(\frac{d}{dt}\right)P(t)=\frac{1}{4\pi}\sum_{k=1}^{2n}(-1)^{k+1}
\frac{1}{\sin^{2} \frac{2k-1}{4n}\pi}P\left(\frac{2k-1}{2n}\pi+t\right), \>\>\>t\in
\mathbb{T}.
\end{equation}
This formula was extended by Boas \cite{B}, \cite{B1}, (see also \cite{Akh}, \cite{Nik}, \cite{Schm}) to functions in the Bernstein class $\mathbf{B}_{\sigma}^{\infty}(\mathbb{R})$ in the following form 
\begin{equation}\label{Boas}
\left( \frac{d}{dt}\right)f(t)=\frac{n}{\pi^{2}}\sum_{k\in\mathbb{Z}}\frac{(-1)^{k-1}}{(k-1/2)^{2}}
f\left(\frac{\pi}{n}(k-1/2)+t\right), \>\>\>t\in \mathbb{R}.
\end{equation}
In turn, the formula (\ref{Boas}) was  extended in \cite{BSS2} to  higher powers   $(d/dt)^{m}$. Our objective in section \ref{RB}   is to obtain  similar formulas for $m\in \mathbb{N}$ where the operator $d/dt$ is replaced by the operator $\Theta=x\frac{d}{dx}$. When $m=1$ such formula for the operator $\Theta$ was established in \cite {BBMS-1}. 

Obviously, the goals of the present article are quite close to some of the objectives of the papers  \cite{BBMS-1}-\cite{BBMS-8}. It would be very interesting and instructive      
to do a rigorous comparison of our approaches  and outcomes. However, the fact that  our papers are based on a rather different ideas makes a such comparison not easy. A serious juxtaposition of our treatments  would  require  substantial increase of the length of the present article. We are planing to do such analysis in a separate paper.

\section{Bernstein spaces}\label{Bern-sp}
\subsection{Mellin translations}

For $p \in  [1, \infty[$, denote by $\| \cdot \|_{p} $ the norm of the Lebesgue space $L ^{p}(\mathbb{R}_{+}).$ In Mellin analysis, the analogue of $L ^{p }(\mathbb{R}_{+} )$ are the spaces $X^{p}(\mathbb{R}_{+})$ comprising all functions $f: \mathbb{R}_{+} \mapsto \mathbb{ C}$  such that $f(\cdot)(\cdot)^{-1/p}\in  L ^{p}(\mathbb{R}_{+}) $ with the norm $\|f\|_{X^{p}(\mathbb{R}_{+}) }:= \| f (\cdot)(\cdot)^{-1/p}\|_{p}$. Furthermore, for $p = \infty$, we define $X ^{\infty}$ as the space of all measurable functions $f : \mathbb{R}_{+} \mapsto\mathbb{C}$ 
such that $\|f\|_{X^{\infty}}:=\sup_{x>0}|f(x)|<\infty$.

In  spaces  $X^{p}(\mathbb{R}_{+})$ we consider the one-parameter $C_{0}$-group of operators 
  $U(t),\>t\in \mathbb{R},$ where
\begin{equation}\label{U}
U(t)f(x)=f(e^{t}x),\>\>\>U(t+\tau)=U(t)U(\tau),
\end{equation}
whose infinitesimal generator is 
\begin{equation}
\frac{d}{dt}U(t)f(x )|_{t=0}=x\frac{d}{dx}f(x)=\Theta f(x).
\end{equation}
The domain of its power $k\in \mathbb{N}$ is denoted as $\mathcal{D}^{k}(\Theta)$ and defined  as the set of all functions $f\in X^{p}(\mathbb{R}_{+}),\>1\leq p\leq \infty$, such that $\Theta^{k}f\in X^{p}(\mathbb{R}_{+})$. The domains $\mathcal{D}^{k}(\Theta),\>\>k\in \mathbb{N}, $ can be treated as analogs of the Sobolev spaces.
The general theory of one-parameter semi-groups of class $C_{0}$ (see \cite{BB}, \cite{K}) implies  that the operator $\Theta$ is closed in $X^{p}(\mathbb{R}_{+})$  and the set $\mathcal{D}^{\infty}(\Theta)=\cap_{k}\mathcal{D}^{k}(\Theta)$ is dense in $X^{p}(\mathbb{R}_{+})$.
By using the following formula  (see \cite{BJ-1}, p. 355)
\begin{equation}
\Theta^{k}f(x)=\sum_{r=0}^{k}S(k, r)x^{r}f^{(r)}(x),
\end{equation}
$S(k,r)$ being Stirling numbers of the second kind, one can give more explicit description of the Mellin-Sobolev spaces (see  \cite{BJ-1}, p. 357) in the spirit of the Bochner's definition of the classical one-dimensional Sobolev spaces.

\subsection{Bernstein spaces}

Let's remind that in the classical analysis a  Bernstein class \cite{Akh}, \cite{Nik}, which is denoted as ${\bf B}_{\sigma}^{p}(\mathbb{R}),\>\> \sigma>0, \>\>1\leq p\leq \infty,$ is a linear space of all functions $f:\mathbb{R} \mapsto \mathbb{C}$ which belong to $L^{p}(\mathbb{R})$ and admit  extension to $\mathbb{C}$ as entire functions of exponential type $\sigma$.
A function $f$ belongs to ${\bf B}_{\sigma}^{p}(\mathbb{R})$ if and only if the following Bernstein inequality holds 
$$
\|f^{(k)}\|_{L^{p}(\mathbb{R})}\leq \sigma^{k}\|f\|_{L^{p}(\mathbb{R})},
$$
for all natural $k$. 
Using the distributional Fourier transform 
$$
\widehat{f}(\xi)=\frac{1}{\sqrt{2\pi}}  \int_{\mathbb{R}} f(x)e^{-i\xi x}dx, \>\>\>f\in L^{p}(\mathbb{R}), \>\>1\leq p\leq \infty,
$$
one can show (Paley-Wiener theorem) that $f\in {\bf B}_{\sigma}^{p}(\mathbb{R}), \>\>1\leq p\leq \infty,$ if and only if $f\in L^{p}(\mathbb{R}),\>\>1\leq p\leq \infty,$ and the support of $\widehat{f}$ (in sens of distributions) is in $[-\sigma, \sigma]$.

\begin{defn}
The Bernstein space subspace
 ${\bf B}_{\sigma}^{p}(\Theta), \>\>\sigma>0,\>1\leq p\leq \infty,$ is defined as  a set of all functions $f$ in $X^{p}(\mathbb{R}_{+})$ 
 which belong to $\mathcal{D}^{\infty}(\Theta)$  and for which 
\begin{equation}\label{Bern}
\|\Theta^{k}f\|_{X^{p}(\mathbb{R}_{+})}\leq \sigma^{k}\|f\|_{X^{p}(\mathbb{R}_{+})}, \>\>k\in \mathbb{N}.
\end{equation}
\end{defn}

\begin{thm}\label{basic}

A function $f\in \mathcal{D}^{\infty}(\Theta)$ belongs to  ${\bf B}_{\sigma}^{p}(\Theta), \>\>\sigma> 0 ,\>1\leq p\leq \infty,$ if and only if
 the quantity
\begin{equation}\label{Lem}
\sup_{k\in N} \sigma ^{-k}\|\Theta^{k}f\|_{X^{p}(\mathbb{R}_{+})}=R(f,\sigma)
\end{equation}
 is finite.
\end{thm}

\begin{proof} It is evident that if  $f\in {\bf B}_{\sigma}^{p}(\Theta)$ then (\ref{Lem}) holds.
Next, for an $h\in X^{q}(\mathbb{R}_{+}),\>1/p+1/q=1,$ consider a scalar-valued function 
$$
\Phi(t)=\int _{\mathbb{R}_{+}}f(e^{t}x)h(x)\frac{dx}{x}.
$$
We note that
$$
\left(x\frac{d}{dx}\right)^{k}f(x)=\left(\frac{d}{dt}\right)^{k}f(e^{t}x)|_{t=0},
$$
and since $f\in \mathcal{D}^{\infty}(\Theta)$ we conclude that $f(e^{t}x)$ is infinitely differentiable at $t=0$.
Now we have that 
\begin{equation}\label{ser}
\Phi(t)=\sum_{k=0}^{\infty}\frac{1}{k!}t^{k}\left(\frac{d}{dt}\right)^{k}\Phi(0)= \sum_{k=0}^{\infty}\frac{1}{k!}t^{k}\int _{\mathbb{R}_{+}}\left(\frac{d}{dt}\right)^{k}f(e^{t}x)|_{t=0}h(x)\frac{dx}{x}=
$$
$$
\sum_{k=0}^{\infty}\frac{1}{k!}t^{k}\int _{\mathbb{R}_{+}}\left(x\frac{d}{dx}\right)^{k}f(x)h(x)\frac{dx}{x}.
\end{equation}
This series is absolutely convergent since by the assumption
\begin{equation}\label{estim}
\left|\Phi(t)\right|\leq \sum_{k=0}^{\infty}\frac{1}{k!}t^{k}\left|\int _{\mathbb{R}_{+}}\Theta^{k}f(x)h(x)\frac{dx}{x}\right|\leq  
$$
$$
R(f,\sigma)\sum_{k=0}^{\infty}\frac{1}{k!}t^{k}\sigma^{k}\|f\|_{X^{p}(\mathbb{R}_{+})}\|h\|_{X^{q}(\mathbb{R}_{+})}= R(f,\sigma)\|f\|_{X^{p}(\mathbb{R}_{+})}\|h\|_{X^{q}(\mathbb{R}_{+})}e^{\sigma t}.
\end{equation}
It implies that $\Phi$ can be extended to the complex plane $\mathbb{C}$ by using its Taylor series (\ref{ser}). Moreover, as the estimate (\ref{estim}) shows  the inequality 
$$
 \left|\Phi(z)\right|\leq R(f,\sigma)\|f\|_{X^{p}(\mathbb{R}_{+})}\|h\|_{X^{q}(\mathbb{R}_{+})}e^{\sigma |z|},\>\>z\in \mathbb{C},
$$
will hold. In addition, $\Phi$ is bounded on the real line by the constant $\|f\|_{X^{p}(\mathbb{R}_{+})}\|h\|_{X^{q}(\mathbb{R}_{+})}$. In other words, we proved \textit{ that if $f\in {\bf B}_{\sigma}^{p}(\Theta),\>h\in X^{q}(\mathbb{R}_{+}),\>1/p+1/q=1,$ then $\Phi $ belongs to the regular Bernstein space $\mathbf{B}_{\sigma}^{\infty}(\mathbb{R})$. }This fact allows to apply to $\Phi$ the classical  Bernstein inequality in the space $C(\mathbb{R})$ of continuous functions on $\mathbb{R}$ with the uniform norm: 
$$
\left|\left(\frac{d}{dt}\right)^{k}\Phi(0)\right|\leq\sup_{t}\left|\left(\frac{d}{dt}\right)^{k}\Phi(t)\right|\leq \sigma^{k}\sup_{t}\left|\Phi(t)\right|.
$$
Since
$$
\left(\frac{d}{dt}\right)^{k}\Phi(0)=\int _{\mathbb{R}_{+}}\left(\frac{d}{dt}\right)^{k}f(e^{t}x)|_{t=0}h(x)\frac{dx}{x}=\int _{\mathbb{R}_{+}}\Theta^{k}f(x)h(x)\frac{dx}{x}
$$
we obtain
$$
\left|  \int _{\mathbb{R}_{+}}\Theta^{k}f(x)h(x)\frac{dx}{x} \right|\leq \sigma^{k}\|f\|_{X^{p}(\mathbb{R}_{+})}\|h\|_{X^{q}(\mathbb{R}_{+})}
$$
Choosing $h$ such that $\|h\|_{X^{p}(\mathbb{R}_{+})}=1$ and 
$$
\int _{\mathbb{R}_{+}}\Theta^{k}f(x)h(x)\frac{dx}{x}=\|\Theta^{k}f\|_{X^{p}(\mathbb{R}_{+})}
$$ 
we obtain
the inequality 
$$
\|\Theta^{k}f\|_{X^{p}(\mathbb{R}_{+})}\leq \sigma ^{k} \|f\|_{X^{p}(\mathbb{R}_{+})}, \>\>k\in \mathbb{N}.
 $$
   Theorem  is proved.
\end{proof}

The following analog of the Paley-Wiener Theorem follows from the proof of the previous theorem.
 
 \begin{thm}\label{PW}
The following conditions are equivalent:

\begin{enumerate}

\item  $f$ belongs to ${\bf B}_{\sigma}^{p}(\Theta),\>1\leq p\leq \infty$;

\item for every $g\in X^{q}(\mathbb{R}_{+}),\>1/p+1/q=1,$ the function 
$$
\Phi(z)=\int_{\mathbb{R}_{+}}f(e^{z}x)g(x)\frac{dx}{x},\>\>z\in \mathbb{C},
$$
belongs to the regular space  $\mathbf{B}_{\sigma}^{\infty}(\mathbb{R})$, i.e. it is an entire function of exponential type $\sigma$ which is bounded on the real line.
\end{enumerate}

\end{thm}

\section{Sampling theorems in Mellin analysis}\label{sampling}

\subsection{A week Shannon type sampling theorem}

Below we are going to use the following known fact (see  \cite{BSS1}, p. 46).

\begin{thm}\label{Shannon-1}

If 
 $h\in {\bf B}_{\sigma}^{\infty}(\mathbb{R})$, then for any $0<\gamma<1$ the following formula holds 

\begin{equation}
h(z)=\sum_{k\in \mathbb{Z}} h\left(\gamma\frac{k\pi}{\sigma}\right)\>
sinc\left(\gamma^{-1}\frac{\sigma}{\pi} z-k\right),\>\>\>z\in \mathbb{C},
\end{equation} 
where the series converges uniformly on compact subsets of $\mathbb{C}$.
\end{thm}

By using Theorem \ref{Shannon-1}  we obtain 
our First "Weak" Sampling Theorem. 

\begin{thm} \label{sampling1}If $f\in \mathbf{B}_{\sigma}^{p}(\Theta),\>\>1\leq p\leq\infty$ then  for all $g\in L^{q}(\Theta), \>\>1/p+1/q=1,$ and all $0<\gamma<1$ the following formula holds

\begin{equation}\label{1-1}
\int_{\mathbb{R}_{+}} f(\tau x)g(x)\frac{dx}{x}=
$$
$$
\sum_{k\in \mathbb{Z}}\left(\int_{\mathbb{R}_{+}} f(e^{\gamma k\pi/\sigma}x)g(x)\frac{dx}{x}\right)sinc\left(\gamma^{-1}\frac{\sigma}{\pi} \ln \tau -k\right),\>\>\>\tau\in \mathbb{R}_{+},
\end{equation}

where the series converges uniformly on compact subsets of $\mathbb{R}_{+}$. 

\end{thm}

\begin{proof} According to Theorem \ref{PW} for any 
 $f\in \mathbf{B}_{\sigma}^{p}(\Theta),\>\>1\leq p\leq \infty$ and any $g\in X^{q}(\mathbb{R}_{+}),\>1/p+1/q=1,$  the function 
\begin{equation}\label{Phi}
\Phi(t)=\int_{\mathbb{R}_{+}} f(e^{t}x)g(x)\frac{dx}{x},\>\>\>t\in \mathbb{R}.
\end{equation}
belongs to $ \mathbf{B}_{\sigma}^{\infty}(\mathbb{R})  $. 
Applying Theorem \ref{Shannon-1} we obtain
\begin{equation}\label{1-2}
\int_{\mathbb{R}_{+}} f(e^{t}x)g(x)\frac{dx}{x}=
$$
$$
\sum_{k\in \mathbb{Z}}\left(\int_{\mathbb{R}_{+}} f\left(e^{\gamma k\pi/\sigma}x\right)g(x)\frac{dx}{x}\right)sinc\left(\gamma^{-1}\frac{\sigma}{\pi} t -k\right),\>\>\>t\in \mathbb{R},
\end{equation}
where the series converges uniformly on compact subsets of $\mathbb{R}$. Setting $\tau=e^{t}$ or $t=\ln \tau$ gives for we obtain: for any $f\in \mathbf{B}_{\sigma}^{p}(\Theta), g\in X^{q}(\mathbb{R}_{+}),\>1/p+1/q=1,$
\begin{equation}\label{1-4}
\int_{\mathbb{R}_{+}} f(\tau x)g(x)\frac{dx}{x}=
$$
$$
\sum_{k\in \mathbb{Z}}\int_{\mathbb{R}_{+}} f\left(e^{\gamma k\pi/\sigma}x\right)g(x)\frac{dx}{x}sinc\left(\gamma^{-1}\frac{\sigma}{\pi} \ln \tau -k\right),\>\>\>\tau\in \mathbb{R}_{+},
\end{equation}
where the series converges uniformly on compact subsets of $\mathbb{R}_{+}$.

Theorem is proved. 
\end{proof}

\subsection{A sampling formula for Mellin convolution} Note, that according to \cite{BJ-1} the Mellin convolution is defined as
$$
F\ast_{\mathcal{M}} G(z)=\int_{\mathbb{R}_{+}} F\left(\frac{z}{u}\right)\overline{G(u)}\frac{du}{u}.
$$
\begin{thm}\label{sampling2}
For any  $f\in \mathbf{B}_{\sigma}^{p}(\Theta)$ and $h\in  X^{q}(\mathbb{R}_{+}),\>\>1/p+1/q=1, \>1<p<\infty$.
the following formula holds
\begin{equation}\label{conv-M}
f\ast_{\mathcal{M}}  h(\tau)=\sum_{k\in \mathbb{Z}}f\ast_{\mathcal{M}}  h\left(\frac{\gamma k\pi}{\sigma}\right)sinc\left(\gamma^{-1}\frac{\sigma}{\pi} \ln \tau -k\right),\>\>\>\tau\in \mathbb{R}_{+},
\end{equation}
where the series converges uniformly on compact subsets of $\mathbb{R}_{+}$.
\end{thm}
\begin{proof}
In the formula (\ref{1-4}) we replace $g(x),\> x>0, $ by $h(1/x),\>x>0,$   and then perform the substitution  $x=1/y$. After all the formula  (\ref{1-4}) takes the form 
\begin{equation}\label{1-5}
f\ast_{\mathcal{M}}  h(\tau)=\int_{\mathbb{R}_{+}} f\left(\frac{\tau }{y}\right)\overline{h(y)}\frac{dy}{y}=
$$
$$
\sum_{k\in \mathbb{Z}}\int_{\mathbb{R}_{+}} f\left(\frac{e^{\gamma k\pi/\sigma}}{y} \right)\overline{h(y)}\frac{dy}{y}sinc\left(\gamma^{-1}\frac{\sigma}{\pi} \ln \tau -k\right)=
$$
$$
\sum_{k\in \mathbb{Z}}f\ast_{\mathcal{M}}  h\left(\frac{\gamma k\pi}{\sigma}\right)sinc\left(\gamma^{-1}\frac{\sigma}{\pi} \ln \tau -k\right),\>\>\>\tau\in \mathbb{R}_{+},
\end{equation}
where the series converges uniformly on compact subsets of $\mathbb{R}_{+}$. 

Theorem is proven.
\end{proof}
\begin{remark}
Note that  (\ref{conv-M}) is an analog of the formula 
\begin{equation}\label{conv}
f\ast g(t)=\sum_{k\in \mathbb{Z}} f\ast g(\gamma k \pi/\sigma)\>sinc\left(\gamma^{-1}\sigma t /\pi-k\right),
\end{equation}
where  $f\in {\bf B}_{\sigma}^{p}(\mathbb{R}),\>\>1\leq p<\infty$, $ \>g\in L^{q}(\mathbb{R}),\>\>1/p+1/q=1,$  $\>\>\>0<\gamma<1$, and 
$$
f\ast g(t)=\frac{1}{\sqrt{2\pi}}\int_{\mathbb{R}}f(x)g(t-x)dx,
$$
is the classical convolution.

\end{remark}

\subsection{Valiron-Tschakaloff-type sampling formulas }

The next theorem contains  an analog   of  the Valiron-Tschakaloff sampling/interpolation formula  \cite{BFHSS}
 in Mellin analysis.

\begin{thm}\label{sampling3}

If $f\in  \mathbf{B}_{\sigma}^{p}(\Theta),\>\>1<p<\infty,$  $x\in \mathbb{R}_{+},\>t\in \mathbb{R}$, then

\begin{equation}\label{VT-1}
   f(\tau) =
   sinc\left(\frac{\sigma}{\pi}  \ln \tau\right)f(1)+\ln \tau \> sinc\left(\frac{\sigma}{\pi}  \ln \tau\right)(\partial_{x} f)(1)+
   $$
   $$
    \sum_{k\in \mathbb{Z}\setminus \{0\}}f\left(e^{k\pi/\sigma}\right)\frac{\sigma}{k\pi}( \ln \tau)  sinc\left(\frac{\sigma}{\pi}  \ln \tau-k\right),\>\>\>\tau\in \mathbb{R}_{+}.
\end{equation}
The series converges asolutely and uniformly on compact subsets of $\mathbb{R}_{+}$.
\end{thm}
\begin{proof}
As we know (see Theorem \ref{PW}), if $f\in  \mathbf{B}_{\sigma}^{p}(\Theta),$ then the function
\begin{equation}\label{}
\Phi(t)=\int_{\mathbb{R}_{+}} f(e^{t}x)\overline{g(x)}\frac{dx}{x},\>\>\>t\in \mathbb{R},
\end{equation}
belongs to $\mathbf{B}_{\sigma}^{\infty}(\mathbb{R})$.
By applying to it the Valiron-Tschakaloff sampling/interpolation formula which holds for functions in $ \mathbf{B}_{\sigma}^{\infty}(\mathbb{R})$ (see \cite{BFHSS})
\begin{equation}\label{VT}
h(t)=t \> sinc\left(\frac{\sigma  t}{\pi}\right)h^{'}(0)+
$$
$$
 sinc\left(\frac{\sigma  t}{\pi}\right)f(0)+\sum_{k\neq 0}\frac{\sigma t}{k\pi} sinc\left(\frac{\sigma  t}{\pi}-k\right)h\left(\frac{k\pi}{\sigma}\right),\>\>\>h\in \mathbb{B}_{\sigma}^{\infty}(\mathbb{R}),
\end{equation}
where convergence is absolute and uniform on compact subsets of $\mathbb{R}$,
we obtain
$$
\Phi(t)=t \> sinc\left(\frac{\sigma  t}{\pi}\right)\Phi^{'}(0)+
$$
$$
 sinc\left(\frac{\sigma  t}{\pi}\right)f(0)+\sum_{k\neq 0}\frac{\sigma t}{k\pi} sinc\left(\frac{\sigma  t}{\pi}-k\right)\Phi\left(\frac{k\pi}{\sigma}\right),
$$
or for $f\in \mathbb{B}_{\sigma}^{p}(\mathbb{R}), 1<p<\infty,$
\begin{equation}\label{vt-0}
\int_{\mathbb{R}_{+}} f(e^{t}x) g(x)\frac{dx}{x} =
  \int_{\mathbb{R}_{+}} \left[sinc\left(\frac{\sigma  t}{\pi}\right)f(x)+t \> sinc\left(\frac{\sigma  t}{\pi}\right)(x\partial_{x} f)(x)\right]g(x)\frac{dx}{x}+
   $$
   $$
    \sum_{k\in \mathbb{Z}\setminus \{0\}}\frac{\sigma t}{k\pi}  sinc\left(\frac{\sigma  t}{\pi}-k\right)\int_{\mathbb{R}_{+}}f(e^{k\pi/\sigma}x)g(x)\frac{dx}{x},
\end{equation}
where convergence is absolute and uniform on compact subsets of $\mathbb{R}$.
Since the following series converges in the norm of $X^{p}(\mathbb{R}_{+})$:
$$
   \left\| \sum_{k\in \mathbb{Z}\setminus \{0\}}\frac{\sigma t}{k\pi}  sinc\left(\frac{\sigma  t}{\pi}-k\right)f(e^{k\pi/\sigma}x)\right\|_{X^{p}(\mathbb{R}_{+})}\leq
   \|f\|_{X^{p}(\mathbb{R}_{+})}\sum_{k\neq 0, \>\sigma t/\pi}\frac{1}{|\sigma t/\pi-k|}\frac{1}{|k|}\leq 
   $$
   $$
   \|f\|_{X^{p}(\mathbb{R}_{+})}\left(\sum_{k\neq \sigma t/\pi}\frac{1}{|\sigma t/\pi-k|^{p}}\right)^{1/p}\left(\sum_{k\neq 0}\frac{1}{|k|^{q}}\right)^{1/q}<\infty,
$$
we can rewrite (\ref{vt-0}) as 
$$
\int_{\mathbb{R}_{+}} f(e^{t}x) g(x)\frac{dx}{x} =
  \int_{\mathbb{R}_{+}} \left[sinc\left(\frac{\sigma  t}{\pi}\right)f(x)+t \> sinc\left(\frac{\sigma  t}{\pi}\right)(x\partial_{x} f)(x)\right]g(x)\frac{dx}{x}+
   $$
   $$
    \int_{\mathbb{R}_{+}} \left[  \sum_{k\in \mathbb{Z}\setminus \{0\}}\frac{\sigma t}{k\pi}  sinc\left(\frac{\sigma  t}{\pi}-k\right)f(e^{k\pi/\sigma}x)\right]g(x)\frac{dx}{x}.
$$
The last equality holds for any $g\in X^{q}(\mathbb{R}_{+}),\>1/p+1/q=1,$ and since for $1<p<\infty$ the space $X^{q}(\mathbb{R}_{+})$ contains all functionals for $X^{p}(\mathbb{R}_{+})$, we can conclude that the following equality holds

\begin{equation}\label{vt-1}
   f(e^{t}x) =
   sinc\left(\frac{\sigma  t}{\pi}\right)f(x)+t \> sinc\left(\frac{\sigma  t}{\pi}\right)(x\partial_{x} f)(x)+
   $$
   $$
    \sum_{k\in \mathbb{Z}\setminus \{0\}}\frac{\sigma t}{k\pi}  sinc\left(\frac{\sigma  t}{\pi}-k\right)f(e^{k\pi/\sigma}x).
\end{equation}
Substituting  $x=1$ into (\ref{vt-1}) we  obtain
\begin{equation}\label{VT-1}
   f(e^{t}) =
   sinc\left(\frac{\sigma  t}{\pi}\right)f(1)+t \> sinc\left(\frac{\sigma  t}{\pi}\right)(\partial_{x} f)(1)+ 
   $$
   $$
   \sum_{k\in \mathbb{Z}\setminus \{0\}}\frac{\sigma t}{k\pi}  sinc\left(\frac{\sigma  t}{\pi}-k\right)f(e^{k\pi/\sigma}),\>\>\>t\in\mathbb{R}.
\end{equation}
For $\tau=e^{t},\>\>t=\ln \tau$, $t\in \mathbb{R}$, one has
\begin{equation}\label{VT-1}
   f(\tau) =
   sinc\left(\frac{\sigma}{\pi}  \ln \tau\right)f(1)+\ln \tau \> sinc\left(\frac{\sigma}{\pi}  \ln \tau\right)(\partial_{x} f)(1)+
   $$
   $$
    \sum_{k\in \mathbb{Z}\setminus \{0\}}f\left(e^{k\pi/\sigma}\right)\frac{\sigma}{k\pi}( \ln \tau)  sinc\left(\frac{\sigma}{\pi}  \ln \tau-k\right),\>\>\>\tau\in \mathbb{R}_{+}.
\end{equation}
Theorem is proven.

\end{proof}
For every $f\in \mathbf{B}_{\sigma}^{p}(\Theta),\>g\in X^{q}(\mathbb{R}_{+}),\>1/p+1/q=1,$
let's introduce the function $\Psi$ defined as follows:

\begin{equation}\label{Psi1}
\Psi(t)=\frac{1}{t}\left(\Phi(t)-\Phi(0)    \right)=\int_{\mathbb{R}_{+}}\frac{f(e^{t}x)-f(x)}{t}g(x)\frac{dx}{x},
\end{equation}
if $t\neq 0$ and 
\begin{equation}\label{Psi2}
\Psi(0)=\frac{d}{dt}\Phi(t)|_{t=0}=\int_{\mathbb{R}_{+}}\Theta f(x)g(x)\frac{dx}{x},
\end{equation}
if $t=0$.

\begin{lemma}\label{Lem01}
If $f\in \mathbf{B}_{\sigma}^{p}(\Theta),\>g\in X^{q}(\mathbb{R}_{+}),\>1/p+1/q=1,$ then 
$\Psi(t)$ defined in (\ref{Psi1}), (\ref{Psi2}) belongs to $\mathbf{B}_{\sigma}^{r}(\mathbb{R})$ for any $r>1$.

\end{lemma}
\begin{proof}
We remind \cite{Nik} that a function $h(t)=\sum_{k}^{\infty}a_{k}t^{k}$ is an entire function of the exponential type $\sigma$ if and only if the the following condition holds 
$$
\overline{\lim}_{k\rightarrow \infty}\sqrt[k]{k!|a_{k}|}\leq \sigma .
$$
Since $\Phi$ belongs to classical Bernstein class $\mathbf{B}_{\sigma}^{\infty }(\mathbb{R})$ (see Theorem \ref{PW}) it is an entire function of the exponential type $\sigma$. Thus $\Phi(t)=\sum_{k}c_{k}t^{k}$ with $\>\>\>\overline{\lim}_{k\rightarrow \infty}\sqrt[k]{k!|c_{k}|}\leq \sigma $. Now we have that $\Psi(t)=\sum_{k} c_{k}t^{k-1}$, where one clearly has   $\>\>\>\overline{\lim}_{k\rightarrow \infty}\sqrt[k]{k!|c_{k+1}|}\leq \sigma $, which means that $\Psi$ is an entire function of the exponential type $\sigma$.  Moreover, for any $r>1$ the function $\Psi$ is in $L^{r}(\mathbb{R})$ since 
$$
\left| \Psi(t)\right|=\left| \int_{\mathbb{R}_{+}} \frac{f(e^{t}x)-f(x)}{t}g(x)\frac{dx}{x}\right|\leq \frac{2\|f\|_{X^{p}(\mathbb{R}_{+})}\|g\|_{X^{q}(\mathbb{R}_{+})}}{|t|}.
$$
Thus $\Psi$ belongs to  $\mathbf{B}_{\sigma}^{r}(\Theta)$ for any $r>1$.

Lemma is proven.

\end{proof}

\begin{thm}\label{sampling4}
If $f\in \mathbf{B}_{\sigma}^{p}(\Theta),\>1<p<\infty,$ then the following sampling formulas hold for $\tau\in \mathbb{R}_{+}$

\begin{equation}\label{s3}
f(\tau)=f(1)+\ln \tau\>(\partial_{x}f)(1) \>\ sinc\left(\frac{\sigma}{\pi} \ln \tau\right)+
$$
$$
\ln \tau \sum_{k\neq 0}\frac{f\left(e^{\frac{k\pi}{\sigma}}
\right)-f(1)}{\frac{k\pi}{\sigma}} \ sinc\left(\frac{\sigma}{\pi} \ln \tau-k\right),
\end{equation}
where the series converges uniformly on compact subsets of $\mathbb{R}_{+}$.

\end{thm}

\begin{proof}
Consider the same $\Psi$ as before. Since $\Psi\in \mathbf{B}_{\sigma}^{r}(\Theta)$ 
it implies (see \cite{BSS1}. p.46) the following formula
$$
\Psi(t)=\sum_{k\in \mathbb{Z}}\Psi\left(\frac{k}{\sigma}\pi\right) sinc\left(\frac{\sigma}{\pi} t-k\right),
$$
the series being uniformly convergent on each compact subset of $\mathbb{R}$.

Thus
$$
 \int_{\mathbb{R}_{+}} \frac{f(e^{t}x)-f(x)}{t}g(x)\frac{dx}{x}=\int_{\mathbb{R}_{+}}\Theta f(x)g(x)\frac{dx}{x}sinc\left(\frac{\sigma}{\pi} t\right)+
 $$
 $$
 \sum_{k\in \mathbb{Z}\setminus \{0\}} \left(\int_{\mathbb{R}_{+}} \frac{f(e^{\frac{k}{\sigma}\pi}x)-f(x)}{\frac{k}{\sigma}\pi}g(x)\frac{dx}{x} \right)sinc\left(\frac{\sigma}{\pi} t-k\right).
$$
The last formula can be rewritten as
$$
\int_{\mathbb{R}_{+}}f(e^{t}x)g(x)\frac{dx}{x}=\int_{\mathbb{R}_{+}}f(x)g(x)\frac{dx}{x}+\left(\int_{\mathbb{R}_{+}}\Theta f(x)g(x)\frac{dx}{x}\right)sinc\left(\frac{\sigma}{\pi} t\right)+
$$
$$
t \sum_{k\in \mathbb{Z}\setminus \{0\}} \left(\int_{\mathbb{R}_{+}} \frac{f(e^{\frac{k}{\sigma}\pi}x)-f(x)}{\frac{k}{\sigma}\pi}g(x)\frac{dx}{x} \right)sinc\left(\frac{\sigma}{\pi} t-k\right).
$$
Since  the series 
$$
\sum_{k\in \mathbb{Z}\setminus \{0\}} \frac{f(e^{\frac{k}{\sigma}\pi}x)-f(x)}{\frac{k}{\sigma}\pi}sinc\left(\frac{\sigma}{\pi} t-k\right),
$$
converges in $X^{p}(\mathbb{R}_{+})$ we can write the equality
$$
\int_{\mathbb{R}_{+}}f(e^{t}x)g(x)\frac{dx}{x}=
$$
$$
\int_{\mathbb{R}_{+}}\left(f(x)+t\Theta f(x)sinc\left(\frac{\sigma}{\pi} t\right)+t\sum_{k\in \mathbb{Z}\setminus \{0\}} \frac{f(e^{\frac{k}{\sigma}\pi}x)-f(x)}{\frac{k}{\sigma}\pi}sinc\left(\frac{\sigma}{\pi} t-k\right)\right)g(x)\frac{dx}{x}.
$$
Because this equality holds true for every $g\in X^{q}(\mathbb{R}_{+}), \>1p+1/q=1, \>1<p<\infty, $ we finally coming to the formula
\begin{equation}\label{s1}
f(e^{t}x)=f(x)+
$$
$$
t(x\partial_{x}f)(x)\> sinc\left(\frac{\sigma t}{\pi}\right)+t\sum_{k\neq 0}\frac{f\left(e^{\frac{k\pi}{\sigma}}
x\right)-f(x)}{\frac{k\pi}{\sigma}} sinc\left(\frac{\sigma t}{\pi}-k\right).
\end{equation}
By setting $x=1$ we obtain
\begin{equation}\label{s2}
f(e^{t})=f(1)+
$$
$$
t(\partial_{x}f)(1) \>sinc\left(\frac{\sigma t}{\pi}\right)+t\sum_{k\neq 0}\frac{f\left(e^{\frac{k\pi}{\sigma}}
\right)-f(1)}{\frac{k\pi}{\sigma}}  sinc\left(\frac{\sigma t}{\pi}-k\right),
\end{equation}
and for $\tau=e^{t},\>t=\ln \tau,$ one has (\ref{s3}). Theorem is proven.

\end{proof}

\section{Two theorems which involve irregular sampling}\label{irreg}

The following fact was proved by J.R. Higgins in \cite{Hig}.
\begin{thm}\label{Hig}
Let $\{t_{k}\}_{k\in \mathbb{Z}}$ be a sequence of real numbers such that 
\begin{equation}\label{Hig1}
\sup_{k\in\mathbb{Z}}|t_{k}-k|<1/4.
\end{equation}
Define the entire function 
\begin{equation}\label{Hig2}
G(z)=(z-t_{0})\prod_{k\in \mathbb{Z}}\left(1-\frac{z}{t_{k}}\right)\left(1-\frac{z}{t_{-k}}\right).
\end{equation}
 Then for all $f\in \mathbf{B}_{\pi}^{2}(\mathbb{R})$ we have
$$
f(t)=\sum_{k\in \mathbb{Z}}f(t_{k})\frac{G(t)}{G^{'}(t_{k})(t-t_{k})},
$$
uniformly on every compact subset of $\mathbb{R}$.
\end{thm}
As we know (Lemma \ref{Lem01}) for every $f\in \mathbf{B}_{\pi}^{p}(\Theta),\>g\in X^{q}(\mathbb{R}_{+}),\>1/p+1/q=1,$ the function $\Psi$ defined as 
\begin{equation}\label{Psi10}
\Psi(t)=\int_{\mathbb{R}_{+}}\frac{f(e^{t}x)-f(x)}{t}g(x)\frac{dx}{x},
\end{equation}
if $t\neq 0$ and 
\begin{equation}\label{Psi20}
\Psi(0)=\int_{\mathbb{R}_{+}}\Theta f(x)g(x)\frac{dx}{x},
\end{equation}
if $t=0$, belongs to $\mathbf{B}_{\pi}^{2}(\mathbb{R})$.
Applying to it Theorem \ref{Hig} we obtain the following.
\begin{thm} Under assumptions and notations of Theorem \ref{Hig},   for every $f\in \mathbf{B}_{\pi}^{p}(\Theta),\>g\in X^{q}(\mathbb{R}_{+}),\>1/p+1/q=1, 1<p<\infty,$ the following formula holds 
\begin{equation}\label{H}
\Psi(t)=\sum_{k\in \mathbb{Z}}\Psi(t_{k})\frac{G(t)}{G^{'}(t_{k})(t-t_{k})},
\end{equation}
 uniformly on every compact subset of $\mathbb{R}$.
\end{thm}
Note, that if the sequence $\{t_{k}\}$ does not contain zero then the formula (\ref{H}) takes the form
$$
\int_{\mathbb{R}_{+}}\frac{f(e^{t}x)-f(x)}{t}g(x)\frac{dx}{x}=\sum_{k\in \mathbb{Z}}
\left(\int_{\mathbb{R}_{+}}\frac{f(e^{t_{k}}x)-f(x)}{t_{k}}g(x)\frac{dx}{x}\right)\frac{G(t)}{G^{'}(t_{k})(t-t_{k})}.
$$
But in the case $t_{0}=0$ the formula (\ref{H}) has the form

$$
\int_{\mathbb{R}_{+}}\frac{f(e^{t}x)-f(x)}{t}g(x)\frac{dx}{x}=\left(\int_{\mathbb{R}_{+}}\Theta f(x)g(x)\frac{dx}{x}\right)\frac{G(t)}{G^{'}(0)t}+
$$
$$\sum_{k\in \mathbb{Z}\setminus \{0\}}
\left(\int_{\mathbb{R}_{+}}\frac{f(e^{t_{k}}x)-f(x)}{t_{k}}g(x)\frac{dx}{x}\right)\frac{G(t)}{G^{'}(t_{k})(t-t_{k})}.
$$

In the paper by C. Seip \cite{S} the following  result can be found.
\begin{thm}\label{S}    Under assumptions and notations of Theorem \ref{Hig},  for any
 $0<\delta<\pi$ and all $f\in \mathbf{B}_{\pi-\delta}^{\infty}(\mathbb{R})$the following holds true
$$
f(t)=\sum_{k\in \mathbb{Z}}f(t_{k})\frac{G(k)}{G^{'}(t_{k})(t-t_{k})},
$$
uniformly on all compact subsets of $\mathbb{R}$.
\end{thm}
This Theorem together with Theorem \ref{PW} imply the following theorem.
\begin{thm}\label{S}
 If $f\in \mathbf{B}_{\pi-\delta}^{p}(\Theta),\>0<\delta<\pi,1\leq p\leq \infty,$
  then 
 under assumptions and notations of Theorem \ref{Hig} 
  the following formula holds uniformly on  compact subsets of $\mathbb{C}$
$$
\Phi(t)=\sum_{k\in \mathbb{Z}}\Phi(t_{k})\frac{G(t)}{G^{'}(t_{k})(t-t_{k})},
$$
or
$$
\int_{\mathbb{R}_{+}}f(e^{t}x)g(x)\frac{dx}{x}=\sum_{k\in \mathbb{Z}}\left(\int_{\mathbb{R}_{+}}f(e^{t_{k}}x)g(x)\frac{dx}{x}\right)\frac{G(t)}{G^{'}(t_{k})(t-t_{k})},
$$
where $g\in X^{q}(\mathbb{R}_{+}),\>1/p+1/q=1, 1\leq p\leq \infty$.
\end{thm}

\section{Riesz-Boas interpolation formulas}\label{RB}

We introduce the following bounded operators in the spaces $X^{p}(\mathbb{R}_{+}),\>1\leq p\leq \infty.$ 
\begin{equation}\label{b1}
\mathcal{R}^{(2m-1)}(\sigma)f(x)=
$$
$$
\left(\frac{\sigma}{\pi}\right)^{2m-1}\sum_{k\in \mathbb{Z}}(-1)^{k+1}A_{m,k}
 f\left(e^{\frac{\pi}{\sigma}(k-1/2)}x\right),\>\> \>\>f\in X^{p}(\mathbb{R}_{+}), \>\>\sigma>0,\>\>\>m\in \mathbb{N},
\end{equation}
and
\begin{equation}\label{b2}
\mathcal{R}^{(2m)}(\sigma)f(x)=
$$
$$
\left(\frac{\sigma}{\pi}\right)^{2m}\sum_{k\in \mathbb{Z}}(-1)^{k+1}B_{m,k}
f\left(e^{\frac{\pi k}{\sigma}}x\right), \>\> \>\>f\in X^{p}(\mathbb{R}_{+}), \>\>\sigma>0,\>\>\>m\in \mathbb{N},
\end{equation}
where $A_{m,k}$ and $B_{m,k}$ are defined 
as 
\begin{equation}\label{A}
A_{m,k}=(-1)^{k+1}  sinc ^{(2m-1)}\left(\frac{1}{2}-k\right)=
$$
$$
\frac{(2m-1)!}{\pi(k-\frac{1}{2})^{2m}}\sum_{j=0}^{m-1}\frac{(-1)^{j}}{(2j)!}\left(\pi(k-\frac{1}{2})\right)^{2j},\>\>\>m\in \mathbb{N},
\end{equation}
for $k\in \mathbb{Z}$,
\begin{equation}\label{B}
B_{m,k}=(-1)^{k+1}  sinc ^{(2m)}(-k)=\frac{(2m)!}{\pi k^{2m+1}}\sum_{j=0}^{m-1}\frac{(-1)^{j}(\pi k)^{2j+1}}{(2j+1)!},\>\>\>m\in \mathbb{N},
\end{equation}
for $k\in \mathbb{Z}\setminus \{0\}$, 
and 
\begin{equation}\label{B0}
B_{m,0}=(-1)^{m+1} \frac{\pi^{2m}}{2m+1},\>\>\>m\in \mathbb{N}.
\end{equation}
Both series converge in $X^{p}(\mathbb{R}_{+}),\>1\leq p\leq \infty,$ and their sums are (see \cite{BSS2})
\begin{equation}
\left(\frac{\sigma}{\pi}\right)^{2m-1}\sum_{k\in \mathbb{Z}}\left|A_{m,k}\right|=\sigma^{2m-1},\>\>\>\>\>
 \left(\frac{\sigma}{\pi}\right)^{2m}\sum_{k\in \mathbb{Z}}\left|B_{m,k}\right|=\sigma^{2m}\label{id-2}.
 \end{equation}
 Since $\|f(e^{t}\cdot)\|_{X^{p}(\mathbb{R}_{+})}=\|f\|_{X^{p}(\mathbb{R}_{+})}$ it implies that 
 \begin{equation}\label{norms}
 \|\mathcal{R}^{(2m-1)}(\sigma)f\|_{X^{p}(\mathbb{R}_{+})}\leq \sigma^{2m-1}\|f\|_{X^{p}(\mathbb{R}_{+})},\
 $$
 $$
 \|\mathcal{R}^{(2m)}(\sigma)f\|_{X^{p}(\mathbb{R}_{+})}\leq \sigma^{2m}\|f\|_{X^{p}(\mathbb{R}_{+})},\>\>\>f\in X^{p}(\mathbb{R}_{+}). 
 \end{equation}
 \begin{thm}For $f\in X^{p}(\mathbb{R}_{+}),\>1\leq p\leq \infty,$ the next two conditions are equivalent:

\begin{enumerate}

\item  $f$ belongs to $\mathbf{B}_{\sigma}^{p}(\Theta), \>\>\sigma>0,\>1\leq p \leq \infty$,

\item the following Riesz-Boas-type interpolation formulas hold true for $r\in \mathbb{N}$

\begin{equation}\label{B1}
\Theta^{r}f=\mathcal{R}^{(r)}(\sigma)f,
\end{equation}
or explicitly
$$
\left(x\frac{d}{dx}\right)^{2m-1}f(x)=
\left(\frac{\sigma}{\pi}\right)^{2m-1}\sum_{k\in \mathbb{Z}}(-1)^{k+1}A_{m,k}
 f\left(e^{\frac{\pi}{\sigma}(k-1/2)}x\right),
$$
and
$$
\left(x\frac{d}{dx}\right)^{2m}f(x)=\left(\frac{\sigma}{\pi}\right)^{2m}\sum_{k\in \mathbb{Z}}(-1)^{k+1}B_{m,k}
f\left(e^{\frac{\pi k}{\sigma}}x\right),
$$
\end{enumerate}
where each of the serious converges absolutely and uniformly on $\mathbb{R}_{+}$.
\end{thm}

\begin{proof}
We are proving that (1)$\rightarrow$ (2). According to Theorem \ref{PW}, if $f\in \mathbf{B}_{\sigma}^{p}(\Theta), \>\>\sigma>0,\>1<p<\infty$, then the function 
 $
\Phi(t)=\int_{\mathbb{R}_{+}}f(e^{t}x)g(x)\frac{dx}{x}
$
for any $g\in X^{q}(\mathbb{R}_{+}),\>1/p+1/q=1,$ 
belongs to $ {\bf B}_{\sigma}^{\infty}(\mathbb{R}).$ Thus by  \cite{BSS2} we have
$$
\Phi^{(2m-1)}(t)=\left(\frac{\sigma}{\pi}\right)^{2m-1}\sum_{k\in \mathbb{Z}}(-1)^{k+1}A_{m,k}\Phi\left(t+\frac{\pi}{\sigma}(k-1/2)    \right),\>\>\>m\in \mathbb{N},
$$
$$
\Phi^{(2m)}(t)=\left(\frac{\sigma}{\pi}\right)^{2m}\sum_{k\in \mathbb{Z}}(-1)^{k+1}B_{m,k}\Phi\left(t+\frac{\pi k}{\sigma}    \right),\>\>\>m\in \mathbb{N}.
$$
Together with 
$$
\left(\frac{d}{dt}\right)^{k}\Phi(t)=\int_{\mathbb{R}_{+}}\Theta f(e^{t}x)g(x)\frac{dx}{x},
$$
  it shows
$$
\int_{\mathbb{R}_{+}}\Theta^{2m-1} f(e^{t}x)g(x)\frac{dx}{x}=
$$
$$
\left(\frac{\sigma}{\pi}\right)^{2m-1}\sum_{k\in \mathbb{Z}}(-1)^{k+1}A_{m,k}\int_{\mathbb{R}_{+}} f\left(e^{\left(t+\frac{\pi}{\sigma}(k-1/2)\right)}x\right)g(x)\frac{dx}{x},\>\>\>m\in \mathbb{N},
$$
and also
$$
\int_{\mathbb{R}_{+}}\Theta^{2m} f(e^{t}x)g(x)\frac{dx}{x}=\left(\frac{\sigma}{\pi}\right)^{2m}\sum_{k\in \mathbb{Z}}(-1)^{k+1}B_{m,k}\int_{\mathbb{R}_{+}} f\left(e^{\left(t+\frac{\pi k}{\sigma}\right)}x\right)g(x)\frac{dx}{x}             ,\>\>\>m\in \mathbb{N}.
$$
Since both series (\ref{b1}) and (\ref{b2}) converge in $X^{p}(\mathbb{R}_{+})$ and the last two equalities hold for any $g\in X^{q}(\mathbb{R}_{+}),\>1/p+1/q=1,$ we obtain the next two formulas
\begin{equation}\label{sam1}
\Theta^{2m-1}f(e^{t}x)=
\left(\frac{\sigma}{\pi}\right)^{2m-1}\sum_{k\in \mathbb{Z}}(-1)^{k+1}A_{m,k}f\left(e^{\left(t+\frac{\pi}{\sigma}(k-1/2)\right)}x\right)           ,\>\>\>m\in \mathbb{N},
\end{equation}
\begin{equation}\label{sam2}
\Theta^{2m}f(e^{t}x)=\left(\frac{\sigma}{\pi}\right)^{2m}\sum_{k\in \mathbb{Z}}(-1)^{k+1}B_{m,k}
f\left(e^{\left(t+\frac{\pi k}{\sigma}\right)}x\right),\>\>\>m\in \mathbb{N}.
\end{equation}
In turn, when $t=0$ these formulas become formulas (\ref{B1}). 
The fact that (2) $\rightarrow$ (1) easily follows from  the formulas (\ref{B1})
and (\ref{norms}).
  
Theorem is proved.
\end{proof}

\begin{col}
If $f$ belongs to  $\mathbf{B}_{\sigma}^{p}(\Theta), 1<p<\infty,$ then for any $\sigma_{1}\geq \sigma,\>\>\>\sigma_{2}\geq \sigma$ one has 
\begin{equation}\label{sigmaB}
\mathcal{R}^{(r)}(\sigma_{1})f=\mathcal{R}^{(r)}(\sigma_{2})f,\>\>\>r\in \mathbb{N}.
\end{equation}
\end{col}
Let us introduce the  notation
 $$
  \mathcal{R}(\sigma)= \mathcal{R}^{(1)}(\sigma).
 $$
 One has the following "power" formula which  follows from the fact that operators $
  \mathcal{R}(\sigma)$ and $\Theta$ commute on any $\mathbf{B}_{\sigma}^{p}(\Theta)$.
 \begin{col}
 For any $r\in \mathbb{N}$ and  any $f\in \mathbf{B}_{\sigma}^{p}(\Theta), 1<p<\infty,$
 \begin{equation}\label{powerB}
 \Theta^{r}f=\mathcal{R}^{(r)}(\sigma)f=\mathcal{R}^{r}(\sigma)f,
 \end{equation}
 where $\mathcal{R}^{r}(\sigma)f=\mathcal{R}(\sigma)\left(...\left(\mathcal{R}(\sigma)\right)\right)f.$
\end{col}
Let us introduce the following notations
$$\mathcal{R}^{(2m-1)}(\sigma,N)f(x)=
\left(\frac{\sigma}{\pi}\right)^{2m-1}\sum_{|k|\leq N}(-1)^{k+1}A_{m,k}f\left(e^{\frac{\pi}{\sigma}(k-1/2)}x\right),
$$
$$
\mathcal{R}^{(2m)}(\sigma, N)f(x)=
\left(\frac{\sigma}{\pi}\right)^{2m}\sum_{|k|\leq N}(-1)^{k+1}B_{m,k}f\left(e^{\frac{\pi k}{\sigma}}x\right).
$$
One obviously has the following set of approximate Riesz-Boas-type formulas. 
\begin{thm}
If $f\in \mathbf{B}_{\sigma}^{p}(\Theta), 1<p<\infty,$ and $r\in \mathbb{N}$ then
\begin{equation}\label{appB}
\Theta^{r}f\approx\mathcal{R}^{(r)}(\sigma, N)f+O(N^{-2}),
\end{equation}
or explicitly
$$
\left(x\frac{d}{dx}\right)^{2m-1}f(x)\approx
\left(\frac{\sigma}{\pi}\right)^{2m-1}\sum_{|k|\leq N}(-1)^{k+1}A_{m,k}
 f\left(e^{\frac{\pi}{\sigma}(k-1/2)}x\right)+O(N^{-2}),
$$
and
$$
\left(x\frac{d}{dx}\right)^{2m}f(x)\approx\left(\frac{\sigma}{\pi}\right)^{2m}\sum_{|k|\leq N}(-1)^{k+1}B_{m,k}
f\left(e^{\frac{\pi k}{\sigma}}x\right)+O(N^{-2}).
$$

\end{thm}

\end{document}